\documentclass[11pt]{amsart}
\usepackage{a4,amsmath,amssymb,amsthm,eucal}

\def\ds{\displaystyle}
\def\e{{\varepsilon}}
\def\div{\mathrm{div}\,}
\def\R{\mathbb{R}}
\def\N{\mathbb{N}}
\def\LL{\mathcal{L}}
\def\RR{\mathcal{R}}

\newtheorem{theorem}{Theorem}[section]
\newtheorem{proposition}[theorem]{Proposition}

\newtheorem{definition}[theorem]{Definition}

\newtheorem{lemma}[theorem]{Lemma}

\theoremstyle{remark}
\newtheorem{remark}[theorem]{Remark}

\title[Weak Optimal Controls for Linear Elliptic Problems]{Weak Optimal Controls in Coefficients for Linear Elliptic Problems}

\author{Giuseppe Buttazzo\qquad\qquad Peter\ I. Kogut}
\address{Giuseppe Buttazzo: Dipartimento di Matematica,
        Universit\`a di Pisa,
        Largo B. Pontecorvo 5,
        56127 Pisa, Italy\quad{\tt buttazzo@dm.unipi.it}}
\address{Peter\ I. Kogut: Department of Differential Equations,
        Dnipropetrovsk National University,
        Naukova str. 13,
        49050 Dnipropetrovsk,
        Ukraine\quad{\tt p.kogut@i.ua}}

\begin{document}

\maketitle

\begin{abstract}
In this paper we study an optimal control problem associated to a linear degenerate elliptic equation with mixed boundary conditions. The equations of this type can exhibit the Lavrentieff phenomenon and non-uniqueness of weak solutions. We adopt the weight function as a control in $L^1(\Omega)$. Using the direct method in the Calculus of variations, we discuss the solvability of this optimal control problem in the class of weak admissible solutions.
\end{abstract}



\section{Introduction}\label{Sec 0}

The aim of this work is to study the existence of optimal solutions in coefficients associated to a linear degenerate elliptic equation with mixed boundary condition. By control variable we mean a weight coefficient in the main part of the elliptic operator. The precise answer of existence or none-existence of an $L^1$-optimal solutions is heavily depending on the class of admissible controls. Here are the main questions: what is the right setting of optimal control problem with $L^1$-controls in coefficients, and what is the right class of admissible solutions to the above problem? Using the direct method in the Calculus of variations, we discuss the solvability of this optimal control problem in the so-called class of weak admissible solutions.

In this paper we deal with an optimal control problem in coefficients for boundary value problem of the form
\begin{equation}\label{0.1}
\left\{
\begin{array}{ll}
-\div\big(\rho(x)\nabla y\big)+y = f &\text{ in }\Omega,\\
y=0 &\text{ on }\Gamma_D,\\
\rho(x)\frac{\partial y}{\partial\nu}=0 &\text{ on }\Gamma_N,
\end{array}
\right.
\end{equation}
where $f\in L^2(\Omega)$ is a given function, the boundary of $\Omega$ is made of two disjoint parts $\partial\Omega=\Gamma_D\cup\Gamma_N$, and $\rho$ is a measurable non negative weight on a bounded open domain $\Omega$ in $\R^N$.

Several physical phenomena related to equilibrium of continuous media are modeled by this elliptic problem. In order to be able to handle with media which possibly are somewhere ``perfect'' insulators or ``perfect'' conductors (see \cite{DauLions}) we allow the weight $\rho$ to vanish somewhere or to be unbounded.

Though numerous papers (see, for instance, \cite{CALDIROLI:00,CHABROWSKI,CIRMI,KovGorb_07,MURTHY,Past,ZhPa1} and references therein) are devoted to variational and non variational approaches to problems related to \eqref{0.1}, only few papers deal with optimal control problems for degenerate partial differential equations (see for example \cite{BOUCHITTE, ButtVarch}). This can be explained by several reasons. Firstly, boundary value problem \eqref{0.1} for locally integrable weight function $\rho$ may exhibit the Lavrentieff phenomenon, the non-uniqueness of weak solutions, as well as other surprising consequences. So, in general, the mapping $\rho\mapsto y(\rho)$ can be multivalued. Besides, the characteristic feature of this problem is the fact that for every admissible control function $\rho$ with properties prescribed above, the weak solutions of \eqref{0.1} belong to the corresponding weighted Sobolev space $W^{1,p}(\Omega,\rho\,dx)$. In addition, even if the original elliptic equation is non-degenerate, i.e. admissible controls $\rho$ are such that $\rho(x)\ge \alpha>0$, the majority of optimal control problems in coefficients have no solution (see for instance \cite{Mu1}).

The optimal control problem we consider in this paper is closely related to the optimal reinforcement of an elastic membrane \cite{ButtVarch}. Reinforcing an elastic structure subjected to a given load is a problem which arises in several applications. The literature on the topic is very wide; for a clear description of the problem from a mechanical point of view and a related bibliography we refer for instance to the beautiful paper by Villaggio \cite{Villaggio}.

In the simplest case when we have an elastic membrane occupying a domain $\Omega$ and subjected to a given exterior load $f\in L^2(\Omega)$, the shape $u$ of the membrane in the equilibrium configuration is characterized as the solution of the partial differential equation
$$-\div\big(\rho(x)\nabla y\big)+y = f\quad\text{in }\Omega$$
together with the corresponding Dirichlet and Neumann boundary conditions on $\partial\Omega$. The reinforcement of the membrane is usually performed by the addition of suitable stiffeners, whose total amount is prescribed. Mathematically, this is described by a nonnegative coefficient $\rho(x)$ which acts in $\Omega$ and is associated with some weight coefficient in the main part of elliptic operator. As a result, the problem of finding an optimal reinforcement for the membrane then consists in the determination of a weight $\rho(x)\ge 0$ which optimizes a given cost functional. In contrast to the paper \cite{ButtVarch}, we do not restrict of our analysis to the particular case of the reinforcement problems. We also do not make use of any relaxations for the original optimal control problem.

\section{Notation and Preliminaries}\label{Sec_1}

Let $\Omega$ be a bounded open subset of $\R^N$ ($N\ge1$) with a Lipschitz boundary. We assume that the boundary of $\Omega$ is made of two disjoint parts
$$\partial\Omega=\Gamma_D\cup\Gamma_N$$
with Dirichlet boundary conditions on $\Gamma_D$, and Neumann boundary conditions on $\Gamma_N$. Let $\chi_E$ be the characteristic function of a subset $E\subset\Omega$, i.e. $\chi_E(x)=1$ if $x\in
E$, and $\chi_E(x)=0$ if $x\not\in E$.

Let $C^\infty_0(\R^N;\Gamma_D)=\left\{\varphi\in C^\infty_0(\R^N)\ :\ \varphi=0\text{ on }\Gamma_D\right\}$. The space $W^{1,1}(\Omega;\Gamma_D)$ is the closure of $C^\infty_0(\R^N;\Gamma_D)$ in
the classical Sobolev space $W^{1,1}(\Omega)$. For any subset $E\subset\Omega$ we denote by $|E|$ its $N$-dimensional Lebesgue measure $\LL^N(E)$.

Let $\rho:\R^N\to\R$ be a locally integrable function on $\R^N$ such that $\rho(x)\ge0$ for a. e. $x\in\R^N$. Then $\rho$ gives rise to a measure on the measurable subsets of $\R^N$ through integration. This measure will also be denoted by $\rho$. Thus $\rho(E)=\int_E\rho\,dx$ for measurable sets $E\subset\R^N$.

We will use the standard notation $L^2(\Omega,\rho\,dx)$ for the set of measurable functions $f$ on $\Omega$ such that
$$\|f\|_{L^2(\Omega,\rho\,dx)}=\Big(\int_\Omega f^2\rho\,dx\Big)^{1/2}<+\infty.$$
We say that a locally integrable function $\rho:\R^N\to\R_+$ is a weight on $\Omega$ if
\begin{equation}\label{1.1}
\rho+\rho^{-1}\in L^1_{loc}(\R^N).
\end{equation}
Note that in this case the functions in $L^2(\Omega,\rho\,dx)$ are Lebesgue integrable on $\Omega$.

To each weight function $\rho$ we may associate two weighted Sobolev spaces:
$$W_\rho=W(\Omega,\rho\,dx)\ \text{ and }\   H_\rho=H(\Omega,\rho\,dx),$$ 
where $W_\rho$ is the set of functions $y\in W^{1,1}(\Omega;\Gamma_D)$ for which the norm
\begin{equation}\label{1.2}
\|y\|_\rho=\Big(\int_\Omega\big(y^2+\rho\,|\nabla y|^2\big)\,dx\Big)^{1/2}
\end{equation}
is finite, and $H_\rho$ is the closure of $C^\infty_0(\Omega;\Gamma_D)$ in $W_\rho$. Note that due to the estimates
\begin{gather}\label{1.2.0}
\int_\Omega|y|\,dx\le\Big(\int_\Omega|y|^2\,dx\Big)^{1/2}|\Omega|^{1/2}\le C\|y\|_\rho,\\
\label{1.2.00}
\int_\Omega|\nabla y|\,dx\le\Big(\int_\Omega |\nabla y|^2 \rho\,dx\Big)^{1/2}\Big(\int_\Omega\rho^{-1}\,dx\Big)^{1/2}\le C\|y\|_\rho,
\end{gather}
the space $W_\rho$ is complete with respect to the norm $\|\cdot\|_\rho$. It is clear that $H_\rho\subset W_\rho$, and $W_\rho$, $H_\rho$ are Hilbert spaces. If $\rho$ is bounded between two positive constants, then it is easy to verify that $W_\rho=H_\rho$. However, for a ``typical'' weight $\rho$ the space of smooth functions $C_0^\infty(\Omega)$ is not dense in $W_\rho$. Hence the identity $W_\rho=H_\rho$ is not always valid (for the corresponding examples we refer to \cite{Chiado_94, Zh_98}).

\bigskip

\textit{Weak Compactness Criterion in $L^1(\Omega)$.} Throughout the paper we will often use the concepts of the weak and strong convergence in $L^1(\Omega)$. Let $\left\{a_\e\right\}_{\e>0}$ be a sequence in $L^1(\Omega)$. We recall that $\left\{a_\e\right\}_{\e>0}$ is called equi-integrable if for any $\delta>0$ there is $\tau=\tau(\delta)$ such that $\int_S |a_\e|\,dx<\delta$ for every measurable subset $S\subset\Omega$ of Lebesgue measure $|S|<\tau$. Then the following assertions are equivalent:
\begin{enumerate}
\item[(i)] a sequence $\left\{a_\e\right\}_{\e>0}$ is weakly compact in $L^1(\Omega)$;
\item[(ii)] the sequence $\left\{a_\e\right\}_{\e>0}$ is equi-integrable;
\item[(iii)] given $\delta>0$ there exists $\lambda=\lambda(\delta)$ such that $\ds\sup_{\e>0}\int_{\{|a_\e|>\lambda\}}|a_\e|\,dx<\delta$.
\end{enumerate}

\begin{theorem}[Lebesgue's Theorem]\label{Th_1.3}
If a sequence $\left\{a_\e\right\}_{\e>0}\subset L^1(\Omega)$ is equi-integ\-rable and $a_\e\to a$ almost everywhere in $\Omega$ then $a_\e\to a$ in $L^1(\Omega)$.
\end{theorem}

\bigskip

\textit{Radon measures.} By a nonnegative Radon measure on $\Omega$ we mean a nonnegative Borel measure which is finite on every compact subset of $\Omega$. The space of all nonnegative Radon measures on $\Omega$ will be denoted by $M_+(\Omega)$. According to the Riesz theory, each Radon measure $\mu\in M_+(\Omega)$ can be interpreted as element of the dual of the space $C_0(\Omega)$ of all continuous functions vanishing at infinity. Let $M(\Omega;\R^N)$ denotes the space of all $\R^N$-valued Borel measures. Then $\mu=(\mu_1,\dots,\mu_N)\in M(\Omega;\R^N)$ $\Leftrightarrow$ $\mu_i\in C^\prime_0(\Omega)$, $i=1,\dots,N$.

If $\mu$ is a nonnegative Radon measure on $\Omega$, we will use $L^r(\Omega,d\mu)$, $1\le r\le\infty$, to denote the usual Lebesgue space with respect to the measure $\mu$ with the corresponding norm $\|f\|_{L^r(\Omega,d\mu)}=\left(\int_\Omega |f(x)|^r\,d\mu\right)^{1/r}$.

\bigskip

\textit{Functions with Bounded Variation}. Let $f:\Omega\to\R$ be a function of $L^1(\Omega)$. Define
\begin{multline*}
\int_\Omega|Df|=
\sup\Big\{\int_\Omega f\div\varphi\,dx\, :\\ \varphi=(\varphi_1,\dots,\varphi_N)\in C^1_0(\Omega;\R^N),\ |\varphi(x)|\le 1\ \text{for}\ x\in \Omega\Big\},
\end{multline*}
where $\div\varphi=\sum_{i=1}^N\frac{\partial\varphi_i}{\partial x_i}$. According to the Radon-Nikodym theorem, if $\int_\Omega|Df|<+\infty$ then the distribution $Df$ is a measure and there exist a vector-valued function $\nabla f\in[L^1(\Omega)]^N$ and a measure $D_s f$, singular with respect to the $N$-dimensional Lebesgue measure $\LL^N\lfloor\Omega$ restricted to $\Omega$, such that
$$Df=\nabla f\LL^N\lfloor\Omega+D_s f.$$

\begin{definition}\label{Def 1.4}
A function $f\in L^1(\Omega)$ is said to have a bounded variation in $\Omega$ if $\int_\Omega|Df|<+\infty$. By $BV(\Omega)$ we denote the space of all functions in $L^1(\Omega)$ with bounded variation.
\end{definition}

Under the norm
$$\|f\|_{BV(\Omega)}=\|f\|_{L^1(\Omega)}+\int_\Omega|D f|,$$
$BV(\Omega)$ is a Banach space. It is well-known the following compactness result for $BV$-functions:

\begin{proposition}\label{Prop 1.4}
The uniformly bounded sets in $BV$-norm are relatively compact in $L^1(\Omega)$.
\end{proposition}

\begin{definition}\label{Def 1.5}
A sequence $\{f_k\}_{k=1}^\infty\subset BV(\Omega)$ weakly converges to some $f\in BV(\Omega)$, and we write $f_k\rightharpoonup f$ iff the two following conditions hold: $f_k\to f$ strongly in $L^1(\Omega)$, and $D f_k\rightharpoonup Df$ weakly* in $M(\Omega;\R^N)$.
\end{definition}

In the proposition below we give a compactness result related to this convergence, together with the lower semicontinuity property (see \cite{Giusti}):

\begin{proposition}\label{Prop 1.6}
Let $\{f_k\}_{k=1}^\infty$ be a sequence in $BV(\Omega)$ strongly converging to some $f$ in $L^1(\Omega)$ and satisfying $\sup_{k\in\N}\int_\Omega|Df_k|<+\infty$. Then
\begin{description}
\item[(i)] $f\in BV(\Omega)$ and $\int_\Omega|Df|\le\liminf_{k\to\infty}\int_\Omega|Df_k|$;
\item[(ii)] $f_k\rightharpoonup f$ in $BV(\Omega)$.
\end{description}
\end{proposition}

\bigskip

\textit{Convergence in variable spaces.}
Let $\{\mu_k\}_{k\in\N}$, $\mu$ be Radon measures such that $\mu_k\stackrel{*}{\rightharpoonup}\mu$ in $M_+(\Omega)$, i.e.,
\begin{equation}
\label{1.10}
\lim_{k\to\infty}\int_\Omega\varphi\,d\mu_k=
\int_\Omega \varphi\,d\mu\qquad\forall\varphi\in C_0(\R^N),
\end{equation}
where $C_0(\R^N)$ is the space of all compactly supported continuous functions. The typical example of such measures is
$$d\mu_k=\rho_k(x)\,dx,\quad d\mu=\rho(x)\,dx,\text{ where }0\le\rho_k\rightharpoonup\rho\text{ in }L^1(\Omega).$$
Let us recall the definition and main properties of convergence in the variable $L^2$-space (see \cite{Zh_98}).
\begin{enumerate}
\item[1.] A sequence $\big\{v_k\in L^2(\Omega,d\mu_k)\big\}$ is called bounded if $\ds\limsup_{k\to\infty}\int_\Omega|v_k|^2\,d\mu_k<+\infty$.

\item[2.] A bounded sequence $\big\{v_k\in L^2(\Omega,d\mu_k)\big\}$ converges weakly to $v\in L^2(\Omega,d\mu)$ if
$$\lim_{k\to\infty}\int_\Omega v_k\varphi\,d\mu_k=
\int_\Omega v\varphi\,d\mu\qquad\forall\varphi\in C^{\infty}_0(\Omega),$$
and it is written as $v_k\rightharpoonup v$ in $L^2(\Omega,d\mu_k)$.

\item[3.] The strong convergence $v_k\to v$ in $L^2(\Omega,d\mu_k)$ means that $v\in L^2(\Omega,d\mu)$ and
\begin{equation}\label{1.10_3}
\lim_{k\to\infty}\int_\Omega v_k z_k\,d\mu_k=
\int_\Omega vz\,d\mu\qquad\text{as }z_k\rightharpoonup z\text{ in }L^2(\Omega,d\mu_k).
\end{equation}
\end{enumerate}
The following convergence properties in variable spaces hold:
\begin{enumerate}
\item[(a)] \textsl{Compactness:} if a sequence is bounded in $L^2(\Omega,d\mu_k)$, then this sequence is compact in the sense of the weak convergence;
\item[(b)] \textsl{Lower semicontinuity:} if $v_k\rightharpoonup v$ in $L^2(\Omega,d\mu_k)$, then
\begin{equation}\label{1.10_0}
\liminf_{\e\to 0}\int_\Omega|v_k|^2\,d\mu_k\ge\int_\Omega v^2\,d\mu;
\end{equation}
\item[(c)] \textsl{Strong convergence:} $v_k\to v$ if and only if $v_k\rightharpoonup v$ in $L^2(\Omega,d\mu_k)$ and
\begin{equation}\label{1.10_1}
\lim_{k\to\infty}\int_{\Omega}|v_k|^2\,d\mu_k=\int_{\Omega}v^2\,d\mu.
\end{equation}
\end{enumerate}


\section{Setting of the Optimal Control Problem}\label{Sec 2}

Let $m\in\R_+$ be some positive value, and let $\xi_1$, $\xi_2$ be given elements of $L^1(\Omega)$ satisfying the conditions
\begin{equation}\label{2.0b}
\xi_1(x)\le\xi_2(x)\text{ a.e. in }\Omega,\quad \xi_1^{-1}\in L^1(\Omega).
\end{equation}
To introduce the class of admissible $BV$-controls we adopt the following concept:

\begin{definition}\label{Def 2.1}
We say that a nonnegative weight $\rho$ is an admissible control to the boundary value problem
\begin{gather}\label{2.2}
-\div\big(\rho(x)\nabla y\big)+y = f\quad\text{in }\ \Omega,\\
\label{2.3}
y=0\text{ on }\Gamma_D,\qquad\rho(x)\frac{\partial y}{\partial\nu}=0\ \text{on }\Gamma_N,
\end{gather}
(it is written as $\rho\in\RR_{ad}$) if
\begin{gather}\label{2.4}
\rho\in BV(\Omega),\quad\int_\Omega\rho\,dx=m,\quad\xi_1(x)\le\rho(x)\le\xi_2(x)\text{ a.e. in }\Omega.
\end{gather}
Here $f\in L^2(\Omega)$ is a given function.
\end{definition}

Hereinafter we assume that the set $\RR_{ad}$ is nonempty.

\begin{remark}\label{Rem 2.5}
In view of the property \eqref{2.0b}, we have the boundary value problem for the degenerate elliptic equation. It means that for some admissible controls $\rho\in\RR_{ad}$ the boundary value problem \eqref{2.2}--\eqref{2.3} can exhibit the Lavrentieff phenomenon, the nonuniqueness of the weak solutions as well as other surprising consequences.
\end{remark}

The optimal control problem we consider in this paper is to minimize the discrepancy between a given distribution $y_d\in L^2(\Omega)$ and the solution of boundary valued problem \eqref{2.2}--\eqref{2.3} by choosing an appropriate weight function $\rho\in\RR_{ad}$. More precisely, we are concerned with the following optimal control problem
\begin{equation}\label{2.6}
\text{Minimize }\left\{I(\rho,y)=\int_\Omega|y(x)-y_d(x)|^2\,dx
+\int_\Omega|\nabla y(x)|_{\R^N}^2\rho\,dx
+\int_\Omega|D\rho|\right\}
\end{equation}
subject to the constraints \eqref{2.2}--\eqref{2.4}.

\begin{definition}
\label{Def 2.7}
We say that a function $y=y(\rho,f)\in W_\rho$ is a weak solution to the boundary value problem \eqref{2.2}--\eqref{2.3} for a fixed control $\rho\in\RR_{ad}$ if the integral identity
\begin{equation}\label{2.8}
\int_\Omega \big(\rho\nabla y\cdot\nabla\varphi+y\varphi\big)\,dx
=\int_\Omega f\varphi\,dx
\end{equation}
holds for any $\varphi\in C^\infty_0(\Omega;\Gamma_D)$.
\end{definition}

It is clear that the question of uniqueness of a weak solution leads us to the problem of density of the subspace of smooth functions $C^\infty_0(\Omega;\Gamma_D)$ in $W_\rho$. However, as was indicated in \cite{ZhPa1}, for a ``typical'' weight function $\rho$ the subspace $C^\infty_0(\Omega;\Gamma_D)$ is not dense in $W_\rho$, and hence there is no uniqueness of weak solutions (for more details and other types of solutions we refer to \cite{Boccardo,KovGorb_07,Zh_98,ZhPa1}). Thus the mapping $\rho\mapsto y(\rho,f)$ is multivalued, in general. Taking this fact into account, we introduce the set
\begin{gather}\label{2.12}
\Xi_W=\left\{(\rho,y)\ \left|\ \rho\in\RR_{ad},\ y\in W_\rho,\ (\rho,y)\text{ are related by \eqref{2.8}}\right.\right\}.
\end{gather}
Note that the set $\Xi_W$ is always nonempty. Indeed, let $V_\rho$ be some intermediate space with $H_\rho\subset V_\rho\subset W_\rho$. We say that a function $y=y(\rho,f)\in V_\rho$ is a $V_\rho$-solution or variational solution to the boundary value problem \eqref{2.2}--\eqref{2.3} if the integral identity \eqref{2.8} holds for every test function $\varphi\in V_\rho$. Hence, in this case the energy equality
\begin{equation}\label{2.12a}
\int_\Omega\left(|\nabla y|^2\rho + y^2\right)\,dx=\int_\Omega fy\,dx
\end{equation}
must be valid. Since
$$\Big|\int_\Omega fy\,dx\Big|\le
\Big(\int_\Omega f^2\,dx\Big)^{1/2}\Big(\int_\Omega y^2\,dx\Big)^{1/2}
\le C\|y\|_\rho$$
for every fixed $f\in L^2(\Omega)$, it follows that the existence and uniqueness of a $V_\rho$-solution are the direct consequence of the Riesz Representation Theorem. Thus every variational solution is also a weak solution to the problem \eqref{2.2}--\eqref{2.3}. Hence $\Xi_W\ne\emptyset$ and therefore the corresponding minimization problem
\begin{equation}\label{2.13}
\inf_{(\rho,y)\in\Xi_W} I(\rho,y)
\end{equation}
is regular. In view of this, we adopt the following concept.

\begin{definition}\label{Def 2.14}
We say that a pair $(\rho^0,y^0)\in L^1(\Omega)\times W^{1,1}(\Omega;\Gamma_D)$ is a weak optimal solution to the problem \eqref{2.4}--\eqref{2.6} if $(\rho^0,y^0)$ is a minimizer for \eqref{2.13}.
\end{definition}


\section{Existence of Weak Optimal Solutions}\label{Sec 3}

Our prime interest in this section deals with the solvability of optimal control problem \eqref{2.4}--\eqref{2.6} in the class of weak solutions. To begin with, we make use of the following results. Let $\left\{(\rho_k,y_k)\in\Xi_W \right\}_{k\in\N}$ be any sequence of weak admissible solutions.

\begin{lemma}\label{Lemma 3.0}
Let $\{\rho_k\}_{k\in\N}$ be a sequence in $\RR_{ad}$ such that $\rho_k\to\rho$ in $L^1(\Omega)$. Then
$$(\rho_k)^{-1}\to\rho^{-1}\ \text{ in the variable space }L^2(\Omega,\rho_k dx).$$
\end{lemma}

\begin{proof}
To prove this result we use some ideas of \cite{ZhPa1}. By the properties of the set of admissible controls $\RR_{ad}$, we have $\rho_k^{-1}\le\xi_1^{-1}$ for every $k\in\N$, hence the sequence $\left\{\rho_k^{-1}\right\}_{k\in\N}$ is equi-integrable on $\Omega$. Note that, up to a subsequence, we have $\rho_k\to\rho$ a.e. in $\Omega$. Since $\xi_2^{-1}\le\rho_k^{-1}\le\xi_1^{-1}$, Lebesgue Theorem implies
\begin{equation}\label{3.00}
\rho_k^{-1}\to\rho^{-1}\text{ in }L^1(\Omega).
\end{equation}
Let $\varphi\in C^\infty_0(\Omega)$ be a fixed function. Then the equality
$$\int_\Omega\rho_k^{-1}\varphi\,\rho_k\,dx=
\int_\Omega\varphi\,dx=
\int_\Omega\rho^{-1}\varphi\,\rho\,dx\quad\forall k\in\N$$
leads us to the weak convergence $\rho^{-1}_k\rightharpoonup\rho^{-1}$ in $L^2(\Omega,\rho_k dx)$. Taking into account the strong convergence $\rho_k^{-1}\to\rho^{-1}$ in $L^1(\Omega)$ and the fact that $\Omega$ is a bounded domain, we get
$$\lim_{k\to\infty}\int_\Omega|\rho_k|^{-2}\rho_k\,dx
=\lim_{k\to\infty}\int_\Omega\rho_k^{-1}\,dx
=\int_\Omega\rho^{-1}\,dx
=\int_\Omega|\rho|^{-2}\rho\,dx.$$
Hence, by the strong convergence criterium in the variable space $L^2(\Omega,\rho_k dx)$, we come to the required conclusion.
\end{proof}

Our next step deals with the study of topological properties of the set of weak admissible solutions $\Xi_W$ to the problem \eqref{2.2}--\eqref{2.6}.

\begin{definition}\label{Def 3.26}
A sequence $\{(\rho_k,y_k)\in\Xi_W\}_{k\in\N}$ is called bounded if
$$\sup_{k\in\N}\left[\|\rho_k\|_{BV(\Omega)}
+\|y_k\|_{L^2(\Omega)}
+\|\nabla y_k\|_{L^2(\Omega,\rho_k dx)^N}\right]
<+\infty.$$
\end{definition}

\begin{definition}\label{Def 3.26a}
We say that a bounded sequence $\{(\rho_k,y_k)\in\Xi_W\}_{k\in\N}$ of the weak admissible solutions $\tau$-converges to a pair $(\rho,y)\in BV(\Omega)\times W^{1,1}(\Omega)$ if
\begin{enumerate}
\item[(a)] $\rho_k\rightharpoonup\rho$ in $BV(\Omega)$;
\item[(d)] $y_k\rightharpoonup y$ weakly in $L^2(\Omega)$;
\item[(e)] $\nabla y_k\rightharpoonup\nabla y\ni L^2(\Omega,\rho\,dx)^N$ in the variable space $L^2(\Omega,\rho_k dx)^N$.
\end{enumerate}
\end{definition}

Note that due to assumptions \eqref{2.0b}, \eqref{2.4}, and estimates like \eqref{1.2.0}--\eqref{1.2.00}, the inclusion $y\in W^{1,1}(\Omega)$ is obvious.

\begin{lemma}\label{Lemma 3.1}
Let $\{(\rho_k,y_k)\in\Xi_W\}_{k\in\N}$ be a bounded sequence. Then there is a pair $(\rho,y)\in BV(\Omega)\times W^{1,1}(\Omega)$ such that, up to a subsequence, $(\rho_k,y_k)\,\stackrel{\tau}{\longrightarrow}(\rho,y)$ and $y\in W_\rho$.
\end{lemma}

\begin{proof}
By Proposition \ref{Prop 1.4} and the compactness criterium of the weak convergence in variable spaces, there exist a subsequence of $\{(\rho_k,y_k)\in\Xi_W\}_{k\in\N}$, still denoted by the same indices, and functions $\rho\in BV(\Omega)$, $y\in L^2(\Omega)$, and $v\in L^2(\Omega,\rho\,dx)^N$ such that
\begin{gather}\label{3.2}
\rho_k\to\rho\text{ in }L^1(\Omega),\\
\label{3.3}
y_k\rightharpoonup y\text{ in }L^2(\Omega),\quad\nabla y_k\rightharpoonup v\text{ in the variable space }L^2(\Omega,\rho_k\,dx).
\end{gather}
Let us show that $y\in W^{1,1}(\Omega)$, and $v=\nabla y$. Since $\xi_1\le\rho_k\le\xi_2$ for every $k\in\N$, \eqref{3.2} and Lemma \ref{Lemma 3.0} imply the property (see \eqref{3.00})
\begin{equation}\label{3.4}
\rho_k^{-1}\to\rho^{-1}\text{ in }L^1(\Omega),\qquad\xi_1\le\rho\le\xi_2\text{ a.e. in }\Omega.
\end{equation}
This yields that the sequence $\{\nabla y_k\}_{k\in\N}$ is weakly compact in $L^1(\Omega)^N$. Indeed, the property of its equi-integrability immediately follows from the inequality
$$\int_E|\nabla y_k|\,dx\le
\Big(\int_E\rho_k^{-1}\,dx\Big)^{1/2}\Big(\int_\Omega|\nabla y_k|^2\rho_k\,dx\Big)^{1/2}
\le C\Big(\int_E\rho_k^{-1}\,dx\Big)^{1/2}.$$
As a result, using the strong convergence $\left(\rho_k\right)^{-1}\to\rho^{-1}$ in the variable space $L^2(\Omega,\rho_k dx)$ (see Lemma \ref{Lemma 3.0}) and its properties, we obtain
\begin{equation*}
\begin{array}{ll}
\ds\lim_{k\to\infty}\int_\Omega\nabla y_k\cdot\psi\,dx=
\lim_{k\to\infty}\int_\Omega \rho_k^{-1}\nabla y_k\cdot\psi\rho_k\,dx\\
\hskip2.9truecm=\ds\int_\Omega\rho^{-1}v\cdot\psi\rho\,dx=\int_\Omega v\cdot\psi\,dx
\end{array}
\end{equation*}
for all $\psi\in C_0^\infty(\Omega)^N$. Thus $\nabla y_k\rightharpoonup v$ in $L^1(\Omega)^N$. This implies that $y\in W^{1,1}(\Omega)$ and $\nabla y=v$. The inclusion $y\in W_\rho$ immediately follows from \eqref{3.2}--\eqref{3.3}.
\end{proof}

\begin{theorem}\label{Th 3.5}
For every $f\in L^2_{loc}(\R^N)$ the set $\Xi_W$ is sequentially closed with respect to the $\tau$-convergence.
\end{theorem}

\begin{proof}
Let $\{(\rho_k,y_k)\}_{k\in\N}\subset\Xi_W$ be a bounded $\tau$-convergent sequence of weak admissible pairs to the optimal control problem \eqref{2.2}--\eqref{2.6}. Let $(\rho_0, y_0)$ be its $\tau$-limit. Our aim is to prove that $(\tau_0,y_0)\in\Xi_W$. By Lemma \ref{Lemma 3.1} we have
\begin{equation}\label{3.6}
\rho_k\to\rho_0\text{ in }L^1(\Omega),\quad\rho_0\in BV(\Omega),\quad\xi_1\le\rho_0\le\xi_2\text{ a.e. in }\Omega.
\end{equation}
Then passing to the limit as $k\to\infty$ in the relation $\int_\Omega \rho_k\,dx=m$, we conclude that $\rho_0\in\RR_{ad}$, i.e. the limit weight function $\rho_0$ is an admissible control.

It remains to show that the pair $(\rho_0,y_0)$ is related by the integral identity \eqref{2.8} for all $\varphi\in C^\infty_0(\Omega;\Gamma_D)$.
For every fixed $k\in\N$ we denote by $(\widehat{\rho}_k,\widehat{y}_k)\in BV_{loc}(\R^N)\times W^{1,1}_{loc}(\R^N)$ an extension of the functions $(\rho_k,y_k)\in\Xi_W$ to the whole of space $\R^N$ such that the sequence $\{(\widehat{\rho}_k,\widehat{y}_k)\}_{k\in\N}$ satisfies the properties:
\begin{gather}
\widehat{\rho}_k\in BV(Q),\quad\xi_1\le\widehat{\rho}_k\le\xi_2\text{ a.e. in }Q,\\
\sup_{k\in\N}\left[\|\widehat{\rho}_k\|_{BV(Q)}
+\|\widehat{y}_k\|_{L^2(Q)}
+\|\nabla\widehat{y}_k\|_{L^2(Q,\widehat{\rho}_kdx)^N}\right]<+\infty
\end{gather}
for any bounded domain $Q$ in $\R^N$. Hence, as done in Lemma \ref{Lemma 3.1} it can be proved that for every bounded domain $Q\subset\R^N$ there exist functions $\widehat{\rho}_0\in BV(Q)$ and $\widehat{y}_0\in W_{\widehat{\rho}_0}$ such that
\begin{gather}\label{3.7}
\widehat{\rho}_k\to\widehat{\rho}_0\text{ in }L^1(Q),\qquad\widehat{y}_k\rightharpoonup\widehat{y}_0\text{ in }L^2(Q),\\
\label{3.8}
\nabla\widehat{y}_k\rightharpoonup\nabla\widehat{y}_0\ni L^2(\Omega,\widehat{\rho}_0\,dx)^N\text{ in the variable space }L^2(\Omega,\widehat{\rho}_k dx)^N.
\end{gather}
It is important to note that in this case we have
\begin{equation}
\label{3.8a}
\widehat{y}_0=y_0\ \text{ and }\ \widehat{\rho}_0=\rho_0\quad\text{a.e. in }\Omega.
\end{equation}
In what follows, we rewrite the integral identity \eqref{2.8} in the equivalent form
\begin{equation}\label{3.9}
\int_{\R^N}\left(\nabla\widehat{y}_k\cdot\nabla\varphi\widehat{\rho}_k
+\widehat{y}_k\varphi\right)\chi_\Omega(x)\,dx
=\int_{\R^N}f\varphi\chi_\Omega(x)\,dx
\qquad\forall\varphi\in C^\infty_0(\R^N,\Gamma_D),
\end{equation}
and pass to the limit in \eqref{3.9} as $k\to\infty$. Using the properties \eqref{3.7}--\eqref{3.8}, and the fact that $\chi_\Omega\to\chi_\Omega$ strongly in the variable space $L^2(Q,\widehat{\rho}_k\,dx)$, i.e.
$$\int_{\R^N}\chi^2_\Omega\widehat{\rho}_k\,dx
=\int_{\R^N}\chi_\Omega\widehat{\rho}_k\,dx
\longrightarrow\int_{\R^N}\chi_\Omega\widehat{\rho}_0\,dx
=\int_{\R^N}\chi^2_\Omega\widehat{\rho}_0\,dx$$
we obtain
$$\int_{\R^N}\left(\nabla\widehat{y}_0\cdot\nabla\varphi\widehat{\rho}_0
+\widehat{y}_0\varphi\right)\chi_\Omega(x)\,dx
=\int_{\R^N}f\varphi\chi_\Omega(x)\,dx\qquad\forall\varphi\in C^\infty_0(\R^N,\Gamma_D)$$
which is equivalent to
$$\int_{\Omega}\left(\nabla\widehat{y}_0\cdot\nabla\varphi\widehat{\rho}_0
+\widehat{y}_0\varphi\right)\,dx
=\int_{\Omega}f\varphi\,dx\qquad\forall\varphi\in C^\infty_0(\Omega,\Gamma_D).$$
As a result, taking into account \eqref{3.8a} and the fact that $\widehat{y}_0\in W_{\widehat{\rho}_0}$ (by Lemma \ref{Lemma 3.1}), we conclude: $y_0$ is a weak solution to the boundary valued problem \eqref{2.2}--\eqref{2.3}  under $\rho=\rho_0$. Thus the $\tau$-limit pair $(\tau_0,y_0)$ belongs to set $\Xi_W$, and this concludes the proof.
\end{proof}

We are now in a position to state the existence of weak optimal pairs to the problem \eqref{2.2}--\eqref{2.6}.

\begin{theorem}\label{Th 3.44}
Let $\xi_1\in L^1_{loc}(\R^N)$ and $\xi_2\in L^1_{loc}(\R^N)$ be such that $\xi_1\le\xi_2$ a.e.in $\R^N$ and $\xi_1^{-1}\in L^1_{loc}(\R^N)$. Let $f\in L^2_{loc}(\R^N)$ and $y_d\in L^2(\Omega)$ be given functions, and assume that $\RR_{ad}\ne\emptyset$. Then the optimal control problem \eqref{2.2}--\eqref{2.6} admits at least one weak solution
$$(\rho^{opt},y^{opt})\in\Xi_W\subset L^1(\Omega)\times W^{1,1}(\Omega,\Gamma_D),
\qquad y^{opt}\in W_{\rho^{opt}}.$$
\end{theorem}

\begin{proof}
Since the set of admissible controls $\RR_{ad}$ is nonempty the minimization problem \eqref{2.13} is regular (i.e. $\Xi_W\ne\emptyset$). Let $\{(\rho_k,y_k)\in\Xi_W\}_{k\in\N}$ be a minimizing sequence to \eqref{2.13}. From the inequality
\begin{equation}
\begin{array}{ll}
\ds\inf_{(\rho,y)\in\Xi_W}I(\rho,y)=\lim_{k\to\infty}\Big[\int_\Omega|y_k(x)-y_d(x)|^2\,dx\\
\ds\hskip2.9truecm+\int_\Omega|\nabla y_k(x)|^2\rho_k\,dx
+\int_\Omega |D\rho_k|\Big]<+\infty,
\end{array}
\end{equation}
there is a constant $C>0$ such that
$$\sup_{k\in\N}\|y_k\|_{L^2(\Omega)}\le C,
\quad\sup_{k\in\N}\|\nabla y_k\|_{L^2(\Omega,\rho_k dx)^N}\le C,
\quad\int_\Omega|D\rho_k|\le C.$$
Hence, in view of the definition of the class of admissible controls $\RR_{ad}$, the sequence $\{(\rho_k,y_k)\in\Xi_W\}_{k\in\N}$ is bounded in the sense of Definition \ref{Def 3.26}. Hence, by Lemma \ref{Lemma 3.1}
there exist functions $\rho^*\in BV(\Omega)$ and $y^*\in W_{\rho^*}$ such that, up to a subsequence, $(\rho_k,y_k)\stackrel{\tau}{\longrightarrow}(\rho^*,y^*)$. Since the set $\Xi_W$ is sequentially closed with respect to the $\tau$-convergence (see Theorem \ref{Th 3.5}), it follows that the $\tau$-limit pair $(\rho^*,y^*)$ is an admissible weak solution to the optimal control problem \eqref{2.2}--\eqref{2.6} (i.e. $(\rho^*,y^*)\in\Xi_W$). To conclude the proof it is enough to observe that by property \eqref{1.10_0} and Proposition \ref{Prop 1.6}, the cost functional $I$ is sequentially lower $\tau$-semicontinuous. Thus
$$I(\rho^*,y^*)\le\liminf_{k\to\infty}I(\rho_k,y_k)
=\inf_{(\rho,\,y)\in\,\Xi_W}I(\rho,y).$$
Hence $(\rho^*,y^*)$ is a weak optimal pair, and we come to the required conclusion.
\end{proof}


\end{document}